\documentclass[11pt]{amsart}

\usepackage[top=4cm, bottom=3cm, left=3cm, right=3cm]{geometry}
\usepackage{hyperref}

\hypersetup{
  pdftitle={Example of an order 16 non symplectic action on a K3 surface},
  pdfauthor={Jimmy Dillies},
  pdfsubject={},
  pdfkeywords={},
  colorlinks=true,
  breaklinks=true,
  bookmarksopen=true,
  bookmarksnumbered=true,
  pdfpagemode=UseOutlines,
  plainpages=false,
  linkcolor=blue,
 citecolor=orange,
  anchorcolor=green}

\newtheorem{thm}{Theorem}[section]
\newtheorem{legend}[thm]{Legend}
\newtheorem{philo}[thm]{Philosophy}
\newtheorem{lem}[thm]{Lemma}

\usepackage{amssymb,amsmath}
\usepackage[applemac]{inputenc}

\usepackage{caption}
\usepackage{subcaption}

\usepackage{graphicx,xcolor}

\newcommand{\IC}{\mathbb{C}}

\newcommand{\IP}{\mathbb{P}}

\title[An Example of Automorphism]{Example of an order 16 non symplectic action on a K3 surface} 

\author{Jimmy Dillies}
\address{Georgia Southern University, Department of Mathematics. Statesboro, GA}
\email{jdillies@georgiasouthern.edu}
\thanks{The author would like to thank Alessandra Sarti for her kind and useful email communication; Fran\c{c}ois Ziegler for patiently probing the unwritten details and an anonymous referee for his recommendations.}
\date{}

\begin{document}


\begin{abstract}
\noindent We exhibit an example of a K3 surface of Picard rank $14$ with a non-symplectic automorphism of order $16$ which fixes a rational curve and $10$ isolated points.  
This settles the existence problem for the last case of Al Tabbaa, Sarti and Taki's classification~\cite{AST}.
\end{abstract}


\maketitle 


\section{Introduction}

Let $X$ be a K3 surface and $\sigma$ an automorphism of $X$. 
The action of $\sigma$ on a volume form on $X$ gives rise to a character $\sigma^{*}$; when this character is non-trivial, we say that the action of $\sigma$ is non-symplectic.
Moreover, we call the action primitive if $\sigma$ and $\sigma^{*}$ have the same order.\\
In~\cite{AST}, D. Al Tabbaa, A. Sarti and S. Taki classify K3 surfaces admitting a primitive non-symplectic automorphism of order 16.
Their classification is complete but for the existence of a K3 surface of Picard lattice $U(2)\oplus D_{4} \oplus E_{8}$ with a non-symplectic automorphism of order $16$ whose fixed locus would consist of $N=10$ isolated points and $k=1$  rational curve.\\
In this brief note, we use our `rigidity' criterion (see~\cite{Di0} or~\cite{Di1}) on symplectic or non-symplectic actions on graphs of rational curves  to construct a model of such a K3 surface and automorphism.\\
We also give a different representation of the action with $N=4$ fixed points and no fixed lines found in Al Tabbaa et al. 
This allows us to show explicitly that both their action and our action can be distinct factorizations of a same non-symplectic automorphism of order $8$.

\begin{legend}
To facilitate a diagonal reading of this article, we list here the conventions used in the Figures below.
{%
\noindent Legend for Figure~\ref{fig:resolution} :  
\protect\includegraphics[height=.8em]{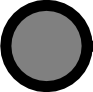} component of the resolution of a singularity of the sextic;
for Figure~\ref{fig:fibrations} : 
\protect\includegraphics[height=.8em]{ball_grey.png} component of singular fiber,  
\protect\includegraphics[height=.8em]{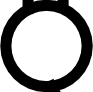} section or multi-section;
for Figures~\ref{fig:chain} - ~\ref{fig:sym}: 
\protect\includegraphics[height=.8em]{ball_grey.png} fixed curve,  
\protect\includegraphics[height=.8em]{ball_white.png} mobile curve,
\protect\includegraphics[height=.8em]{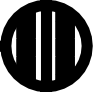} stable curve, 
\protect\includegraphics[height=.8em]{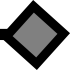} fixed point.
}
\end{legend}

\section{Premises}

In~\cite{AST}, it is shown that if a K3 surface of Picard rank $14$ admits a non-symplectic automorphism $\sigma$ of order $16$ with a fixed locus consisting of $N=10$ isolated points and $k=1$ fixed rational curve, then the Picard lattice of the surface is isomorphic to $U(2)\oplus D_{4} \oplus E_{8}$.
Moreover, to illustrate another type of action (see Section~\ref{sec:appendix}), the authors provide a model of such a surface:\\
Define $S$ as the reducible planar sextic of equation
$$
x_{0}\left( x_{0}^{4}x_{2} + x_{1}^{5} - 2 x_{1}^{3} x_{2}^{2} + x_{2}^{4}x_{1} \right) = 0.
$$
The sextic consists of a line $L$ of equation $x_{0}=0$ and a quintic $C$ described by remaining factor in the above equation.
The surface with Picard lattice  $U(2)\oplus D_{4} \oplus E_{8}$ is the resolution of the double cover of the plane branched along $S$.

\section{The Picard lattice}

Given the Picard lattice, one can rebuild the incidence graph of smooth rational curves on a K3 surface.
This was explained, and worked out explicitly for the above lattice, by S.-M. Belcastro in~\cite{Bel}. 
We represent this graph in Figure~\ref{fig:resolution}.
Part of the rational curves on the above surface are the exceptional curves that appear when blowing up the two $D_{6}$ and the $A_{1}$ singularities which lie at the intersection of $L$ and $C$. 
Given that these three configurations are disjoint, we can identify them unambiguously on the above graph --
the vertices corresponding to these curves are colored in grey on the graph.
Furthermore, we can now identify $L$ on the graph: it is the lone curve which intersects all three singular configurations:

\begin{figure}[!htb]

\centering
\includegraphics[height=2cm]{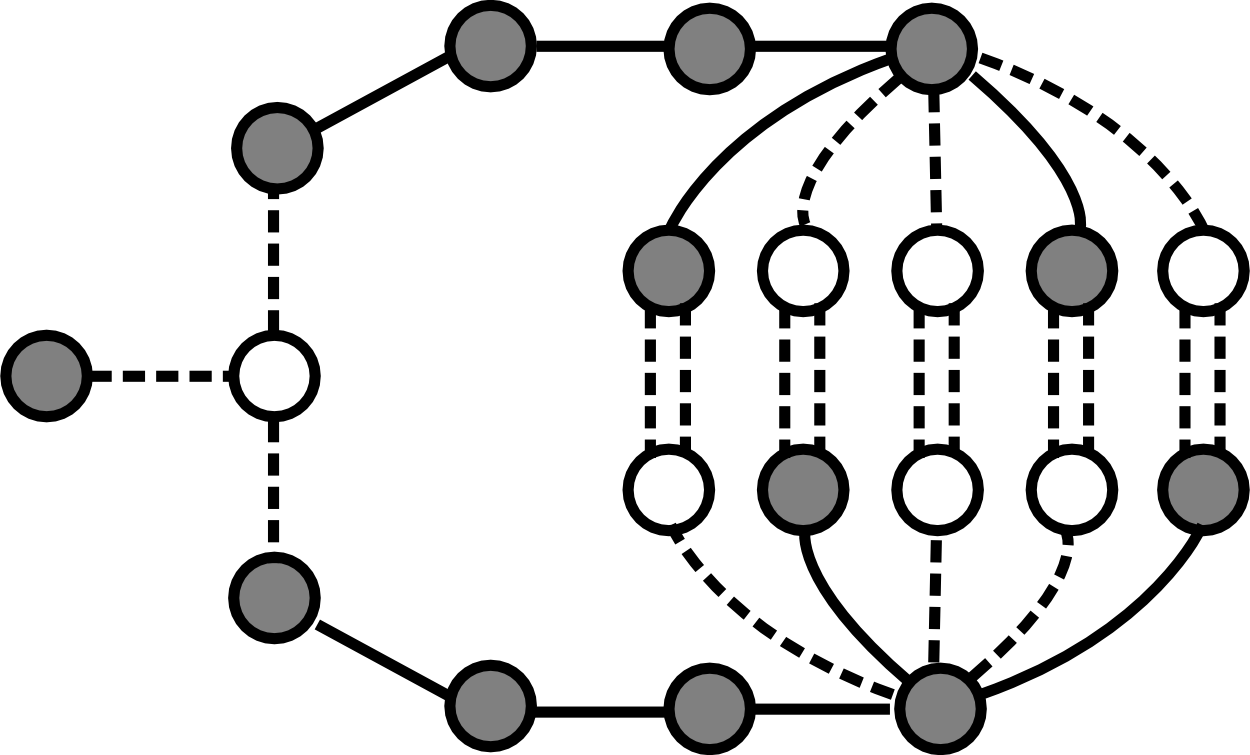}

\caption{Exceptional curves coming from blowing-up along the sextic $S$.}
\label{fig:resolution}

\end{figure}

Our argument will rely heavily on the above graph.
Before we move on, remark that, as in Figure~\ref{fig:fibrations}, we can decompose the graph in different ways.
These decompositions come in parallel with distinguished types of genus $1$ fibrations on our surface.

\begin{figure}[!htb]
       
       \centering
        
        \begin{subfigure}[b]{0.3\textwidth}
              \includegraphics[height=2cm]{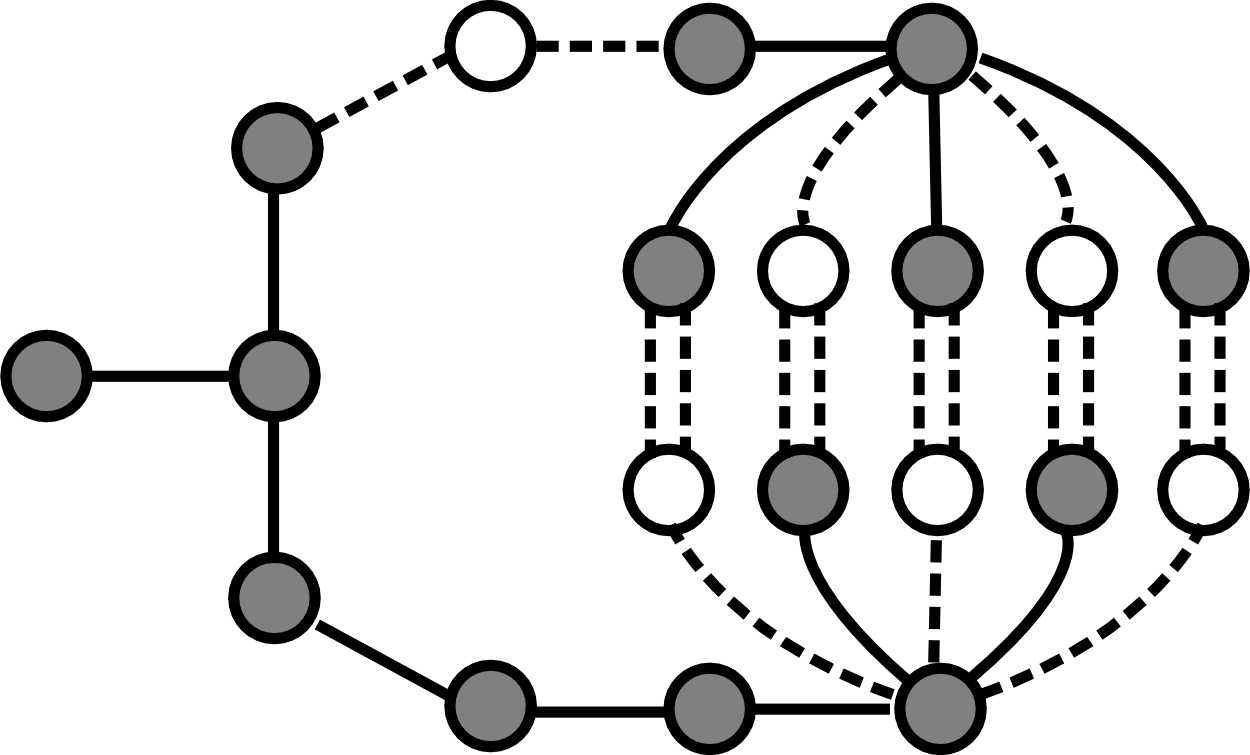}
		\caption{An $I_0^*$ and an $I_4^*$ fiber, five 2-sections and one section}
                \label{fig:lattice1}
        \end{subfigure}%
        ~
        \begin{subfigure}[b]{0.3\textwidth}
                \includegraphics[height=2cm]{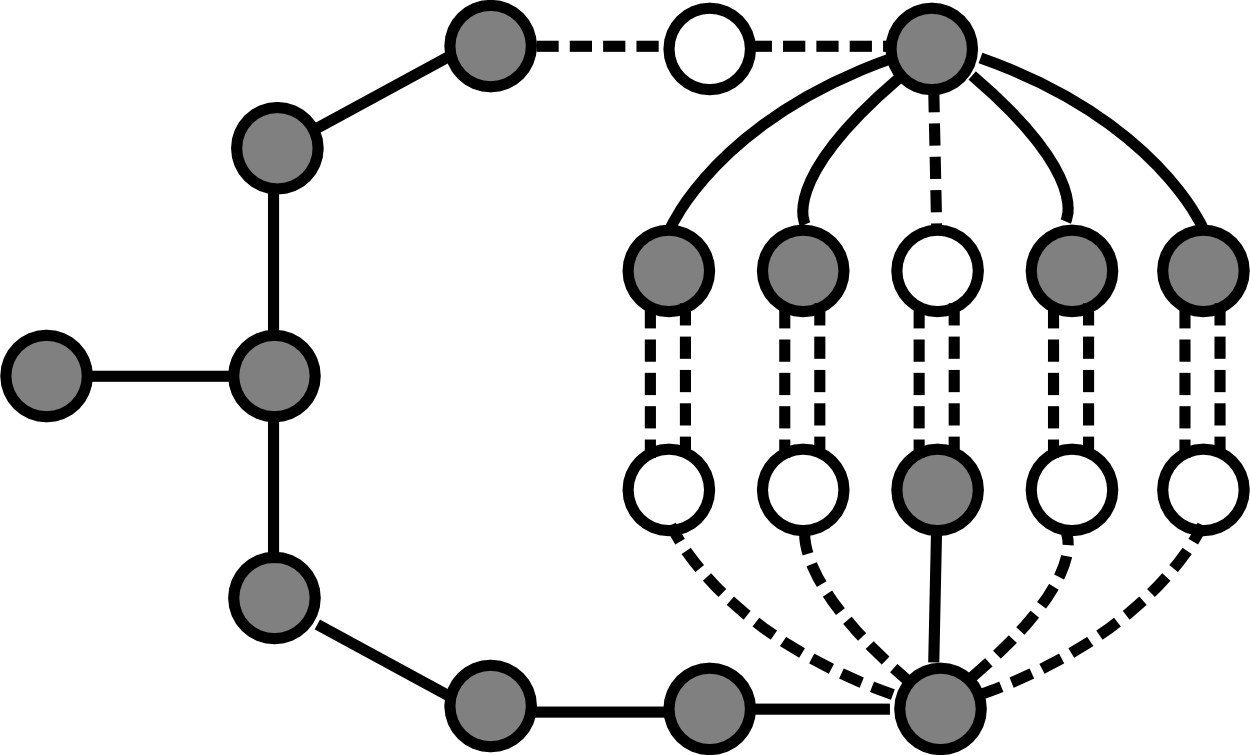}
		\caption{An $I_0^*$ and an $II^*$ fiber, six 2-sections}
   		\label{fig:lattice2}
        \end{subfigure}
        ~ 
        \begin{subfigure}[b]{0.3\textwidth}         
	\includegraphics[height=2cm]{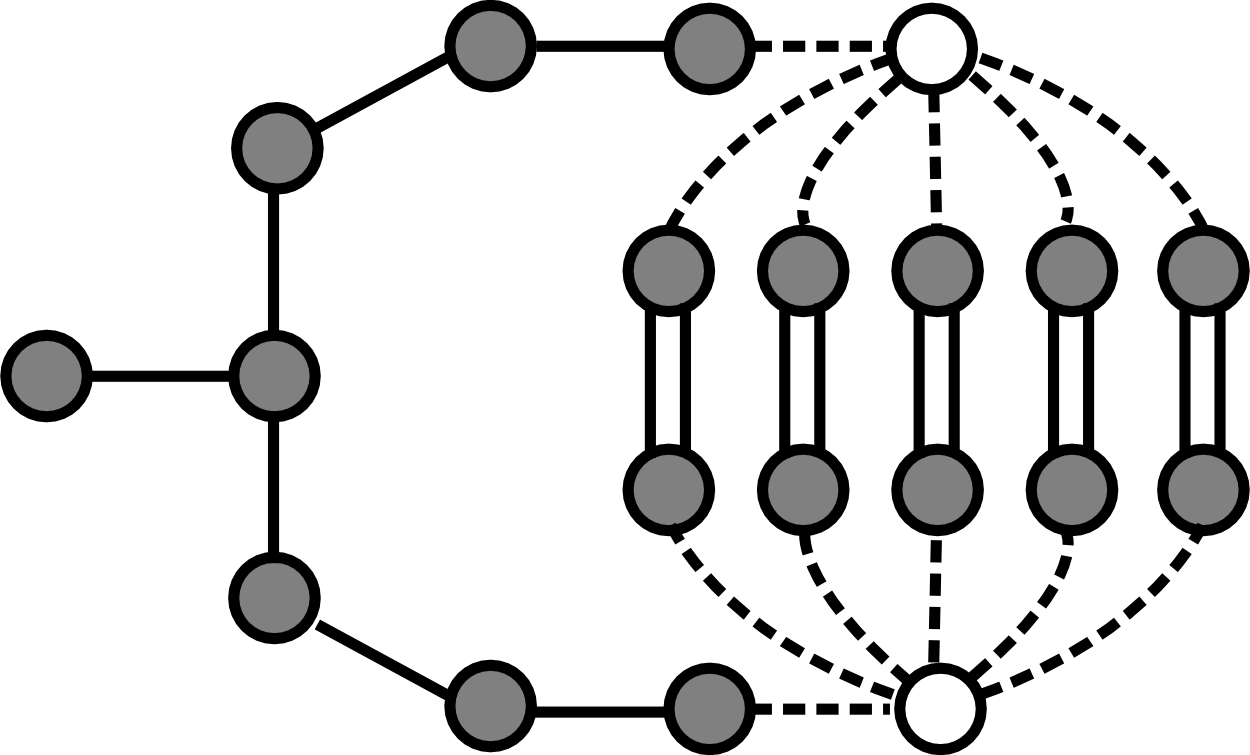}
	\caption{A $III^*$ and five $III$ fibers, two sections}   
	\label{fig:sub3}
        \end{subfigure}
        
	\caption{Lattice decompositions of $U\oplus D_{8} \oplus D_{4} \simeq U(2) \oplus E_{8} \oplus D_{4}$. }
	\label{fig:fibrations}
	
\end{figure}

\section{Building models from stick figures}

Let us start by reminding our 'rigidity' principle - which we obtain \emph{mutatis mutandis} from ~\cite[Lemma 8.1]{Di1}: 
\begin{philo}
The action of a non-symplectic group on a configuration of rational curves is essentially determined by the action at a single point.
\end{philo}

More concretely, the action at two transverse rational curves is constrained by the action on the volume form, and the action at two fixed points on a given curve is constrained by the nature of automorphisms of $\IP^1$.
In particular, for a primitive non-symplectic automorphism of order $16$, in a chain of transverse rational curves, we get the action described in Figure~\ref{fig:chain_lines}.

\begin{figure}[!htb]
\centering
\includegraphics[height=2cm]{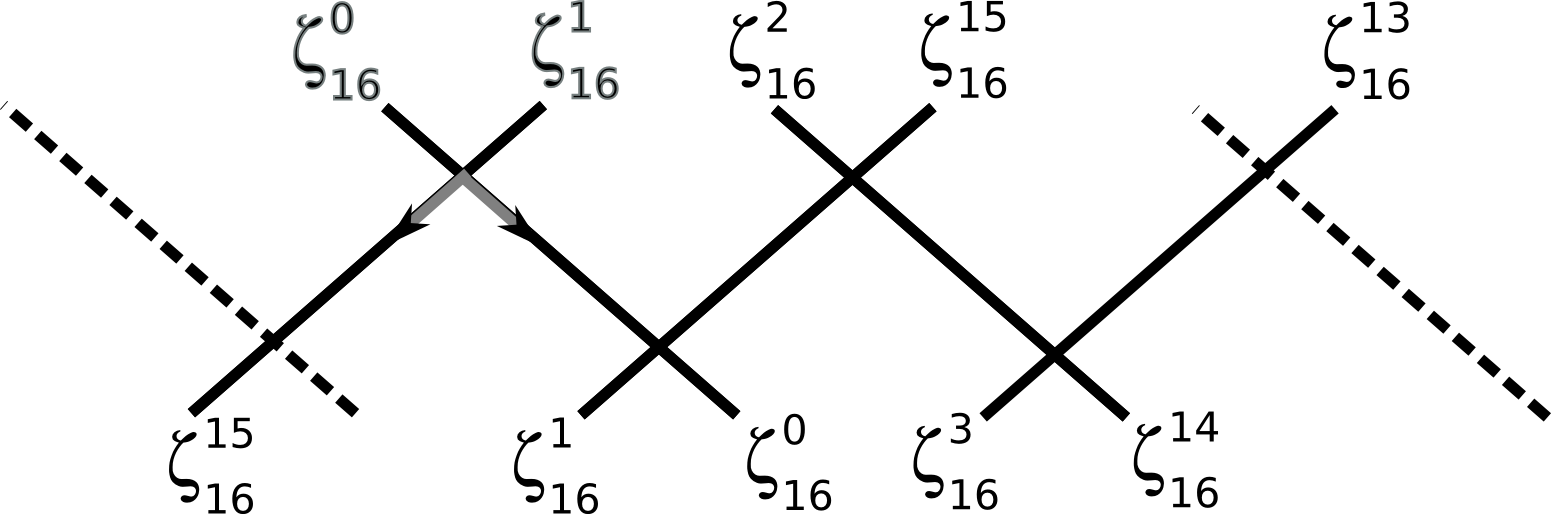}
\caption{Chain of rational curves.}
\label{fig:chain_lines}
\end{figure}

The annotation $\zeta_{16}^k$ means that along that rational curve, around the point of intersection, the action looks locally like $z\mapsto \zeta_{16}^k z$.
In accordance with the legend in the Introduction, we stylise this action as  in Figure~\ref{fig:chain}.

\begin{figure}[!htb]
\centering
\includegraphics[height=1em]{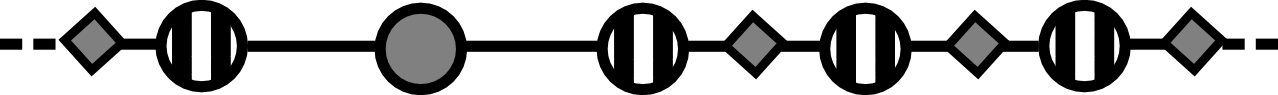}
\caption{Stylized chain of rational curves.}
\label{fig:chain}
\end{figure}

Note that we do not mark the local action at the intersection points.
This can be obtained in a trice from the observation that along the fixed curve, marked in grey, the associated action is $z\mapsto \zeta_{16}^0 z$.\\

Let us now analyze the situation when the fixed locus consists of $1$ rational curve and $10$ isolated points. 

\begin{lem}
\label{lem:action}
There is, up to isometry, a unique order $16$ action on the $U(2) \oplus D_{4} \oplus E_{8}$ configuration of rational curves that fixes $1$ rational curve and $10$ points.
This action is described in Figure~\ref{fig:action}.
\end{lem}

\begin{proof}
Let us first focus on the action restricted to $L$. \\
The automorphism either permutes or stabilizes the upper and lower part of the graph.
Since we need $10$ fixed points, the only possibility is that the upper and the lower branch are stable.
From the symmetry of the graph, the curve coming from the resolution of the $A_1$ singularity is also stable.
We have thus $3$ fixed points on $L$, which means that the action is trivial on $L$.
Using our rigidity philosophy, we can deduce the action on the whole graph.
In particular, on the rational curve corresponding to the vertices of degree $6$, the action has the form $z\mapsto \zeta_{16}^4 z$, i.e. it is of order $4$.
This implies that $4$ of the pairs of rational curves are permuted and $1$ is stable.
\end{proof}

\begin{figure}[!htb]

\centering
\includegraphics[height=2cm]{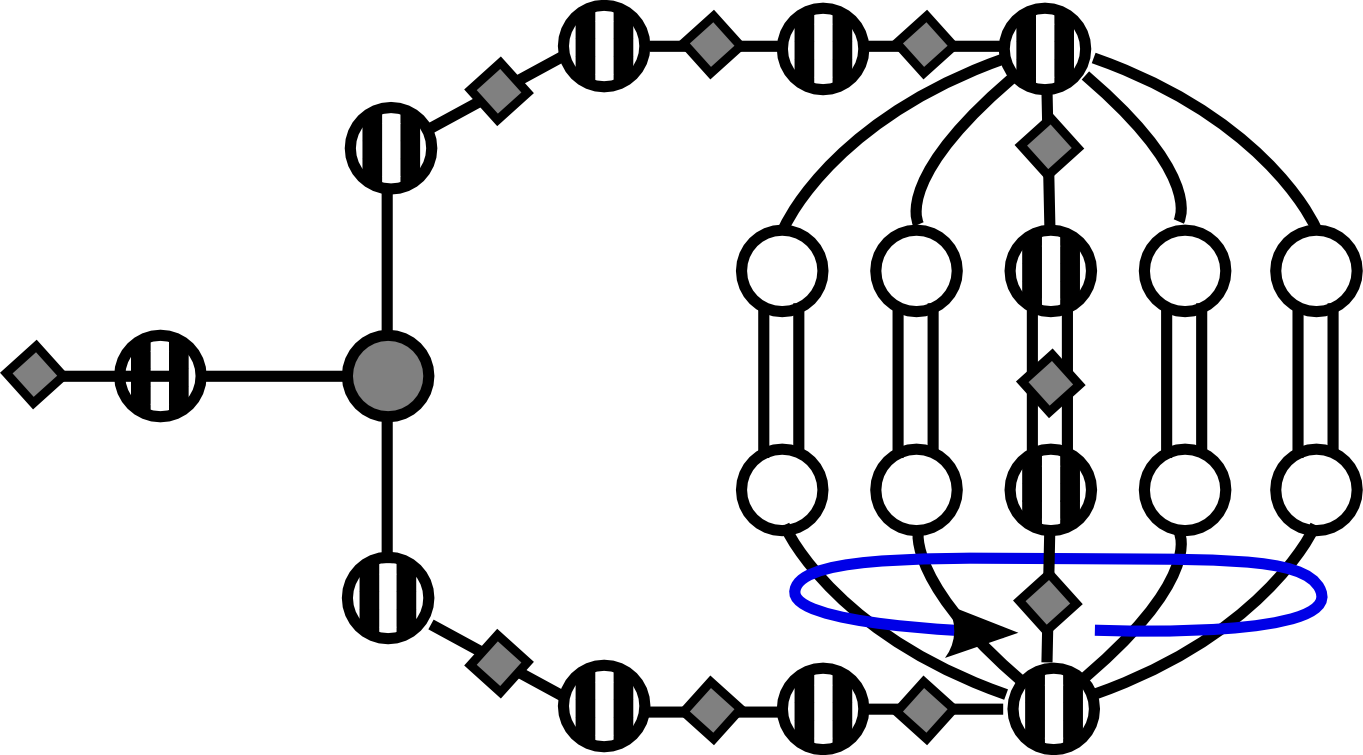}

\caption{Description of the action}
\label{fig:action}

\end{figure}

We are now ready to describe an explicit model of the sought after automorphism.


\subsection{A suitable Weierstrass model}

From Lemma~\ref{lem:action} we see that the action preserves the $III^{*}$ $\left(\tilde{E}_7\right)$ fiber and hence the fibration described in Figure~\ref{fig:sub3}.
We also see that the action on the sections is of order $4$, whence it is of the same order on the base.
Given that the fibration has a singular fiber of type $III^{*}$ and five fibers of type $III$, we can use the discriminantal vanishing criteria, as found in Miranda~\cite[Table IV.3.1]{Mi}, to write down the associated model:

$$
\left\{
\begin{array}{rcl}      
   y^{2} &=  &x^{3} + t^{3}(t^{4}-1) x \\
   \sigma &:   &(x,y,t) \mapsto (\zeta_{16}^{6} x, \zeta_{16}^{9} y, \zeta_{16}^{4} t)
\end{array}\right.
$$

One sees directly that the action is primitive (that is the induced action on $H^{2,0}(X,\IC)$ is also of order $16$) from the explicit equation of the volume form: $\frac{dx \wedge dt}{y}$.

\section{Relation to the other non-symplectic action}

\label{sec:appendix}

Figure~\ref{fig:sarti} schematizes the action, on the same surface as above, with $N=4$ isolated fixed points and $k=0$ fixed rational curves described by D. Al Tabbaa et al. in op. cit.

\begin{figure}[!htb]
\centering
\includegraphics[height=2cm]{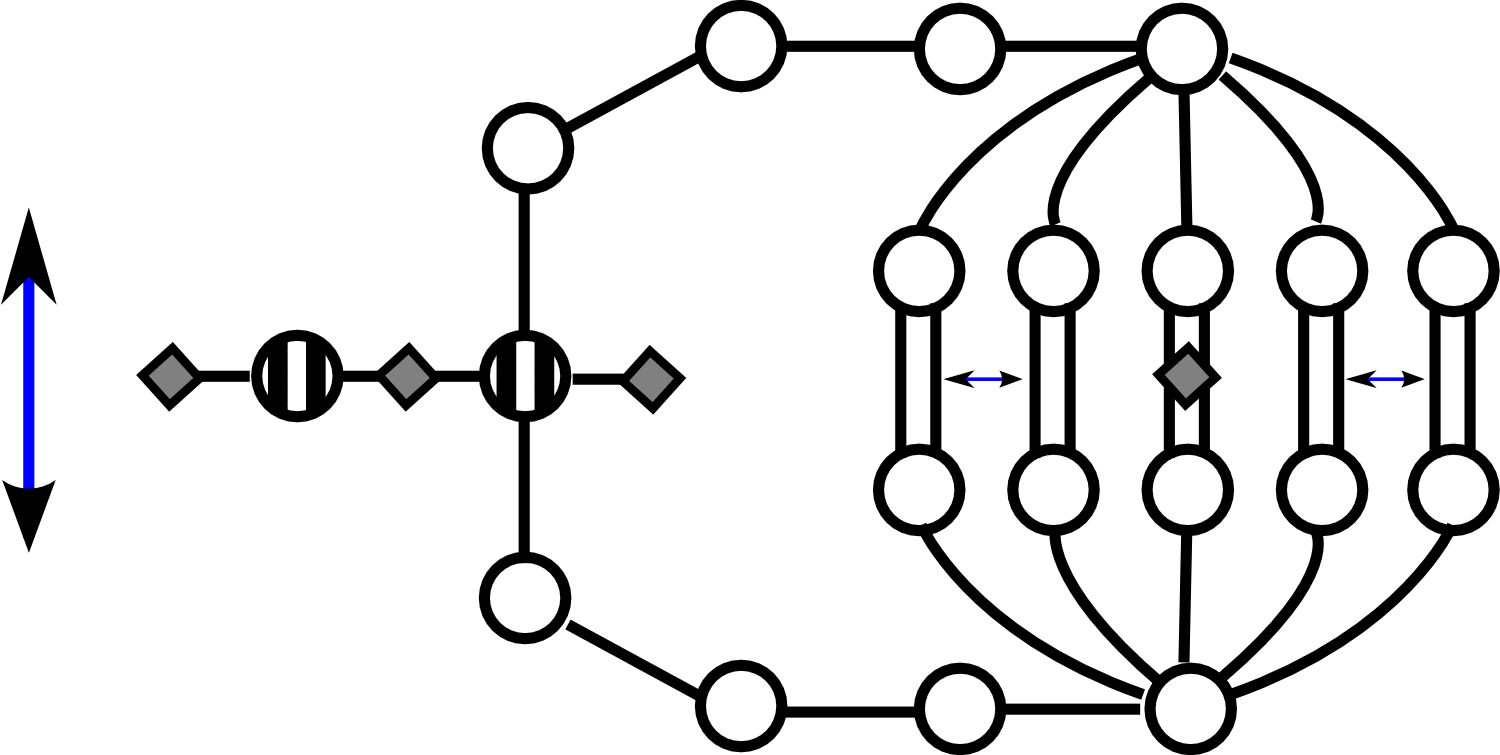}
\caption{Alternative non-symplectic action as in Al Tabbaa et al.}
\label{fig:sarti}
\end{figure}

As for the previous action, the $\tilde{E}_{7}$ fiber is preserved so we can expect to have a representation of this $N=4$, $k=0$ automorphism as a Weierstrass model.
Indeed, we have

$$
\left\{
\begin{array}{rcl}      
   y^{2} &=  &x^{3} + t^{3}(t^{4}-1) x \\
   \sigma_{AST} &:   &(x,y,t) \mapsto \left(\zeta_{16}^{6} \frac{y^{2}-x^{3}}{x^{2}}, \zeta_{16}^{9} \frac{x^{3}y-y^{3}}{x^{3}}, \zeta_{16}^{4} t \right)
\end{array}\right.
$$

One recognizes in the action on the fibers the composition of the map $(x,y)\mapsto (\zeta_{16}^{6} x,\zeta_{16}^{9}y)$ and of the affine representation of the translation by the $2$ torsion section $(0,0)$ : $(x,y)\mapsto (\frac{y^{2}-x^{3}}{x^{2}}, \frac{x^{3}y-y^{3}}{x^{3}})$.
Since the two torsion section is fixed by the automorphism of order $4$, these two actions commute. 

\par

Al Tabbaa et al., showed in their work how the fixed locus of $\sigma^2$ and $\sigma_{AST}^{2}$ are of the same type (i.e. 
$N_{\sigma^{2}}=10$ and $k_{\sigma^{2}}=1$).
What the Weierstrass models above show is that it is possible for the two automorphisms to be different factorizations of the same automorphism of order $8$, i.e.
$\sigma_{AST}^{2}=\sigma^{2}$.

\par 

Finally, let us make the following observation. 
Through composition of $\sigma$ and $\sigma_{AST}^{-1}$, we obtain a Nikulin involution (i.e. a symplectic involution):
$$
\tau_{\mathrm{symp}}:=\sigma \circ \sigma_{AST}^{-1}=\left(\frac{y^{2}-x^{3}}{x^{2}}, \frac{x^{3}y-y^{3}}{x^{3}}, t\right)
$$
And as we can see from our rigidity principle - by adding the weights of the respective actions at the germs around the fixed points - the action of $\tau$ fixes $8$ points (see Figure~\ref{fig:sym}) which is exactly the number predicted in V. Nikulin's work (see ~\cite{Ni,Mo}).

\begin{figure}[!htb]
\centering
\includegraphics[height=2cm]{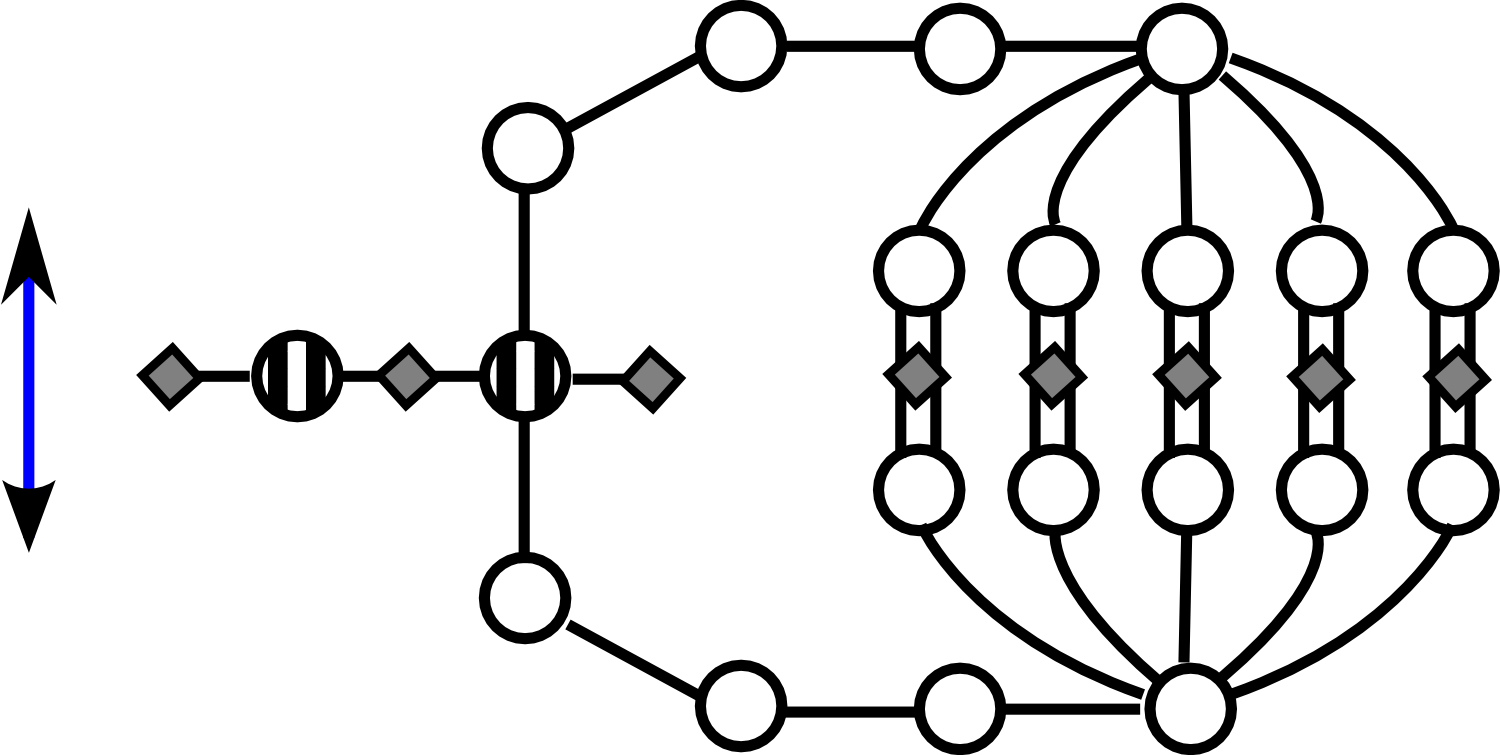}
\caption{Symplectic action of $\tau$.}
\label{fig:sym}
\end{figure}



\begin{thebibliography}{50}

\bibitem{AST}
D. Al Tabbaa, A.Sarti, and S.Taki.
Classification of order sixteen non-symplectic automorphisms on K3 surfaces.
\texttt{arXiv:1409.5803}	


\bibitem{Bel}
S.-M. Belcastro.
Picard lattices of families of K3 surfaces.
Commun. Algebra 30 (2002), 61-82, 


\bibitem{Di0}
J. Dillies.
Automorphisms and Calabi-Yau Threefolds. 
PhD thesis, University of Pennsylvania (2006)


\bibitem{Di1}
J. Dillies.
On some order 6 non-symplectic automorphisms of elliptic K3 surfaces.
Albanian Journal of Mathematics 6 (2012), 103-114


\bibitem{Mi}
R. Miranda. 
The Basic Theory of Elliptic Surfaces. 
ETS Editrice Pisa.

\bibitem{Mo}
 D.R. Morrison. 
 On K3 surfaces with large Picard number. 
 Invent. Math. 75 (1986) 105 -121.

\bibitem{Ni}
V.V. Nikulin.
Kummer surfaces.
Izvestia Akad. Nauk SSSR 39 (1975), 278-293.

\end{thebibliography}
\end{document}